\documentclass[a4paper,11pt,twoside]{article}

\usepackage{latexsym,amsmath,amssymb}\usepackage{amsthm,amscd,mathrsfs}
\usepackage{bm}
\usepackage[dvipdfmx]{graphicx}
\usepackage{fancybox}
\usepackage{multicol}
\usepackage{tabularx}
\usepackage{comment}
\usepackage[top=3.5cm,bottom=3.5cm,left=3cm,right=3cm]{geometry}

\newtheorem{thm}{Theorem}[section]
\newtheorem{prop}{Proposition}[section]
\newtheorem{cor}{Corollary}[section]
\newtheorem{lem}{Lemma}[section]

\newtheorem{df}{Definition}[section]

\theoremstyle{definition}

\newtheorem{rem}{Remark}[section]

\newcommand{\lesim}{\protect\raisebox{-1.0ex}{$\:\stackrel{\textstyle <}{\sim}\:$}} 

\makeatletter
 
 \@addtoreset{equation}{section}
\makeatother

\title{{\bf {\large ANOTHER APPLICATION OF DILATION ANALYTIC METHOD FOR COMPLEX LIEB--THIRRING TYPE ESTIMATES}}}
\author{{\sc Norihiro Someyama}}
\date{\empty}
\begin{document}
\maketitle
\markboth{N. Someyama}{Another Application of Dilation Analytic Method for Complex LT-type Estimates}

\begin{abstract}
We consider non-self-adjoint Schr\"{o}dinger operators $H_{{\rm c}}=-\Delta+V_{{\rm c}}$ (resp. $H_{{\rm r}}=-\Delta+V_{{\rm r}}$) acting in $L^2(\mathbb R^d)$, $d\ge 1$, with dilation analytic complex (resp. real) potentials.
We were able to find out perhaps a new application of dilation analytic method in \cite{So1} (N. Someyama, "Number of Eigenvalues of Non-self-adjoint Schr\"{o}dinger Operators with Dilation Analytic Complex Potentials," Reports on Mathematical Physics, Volume 83, Issue 2, pp.163-174 (2019).).
We give a Lieb--Thirring type estimate on resonance eigenvalues of $H_{{\rm c}}$ in the open complex sector and that on embedded eigenvalues of $H_{{\rm r}}$ in the same way as \cite{So1}.
To achieve that, we derive Lieb--Thirring type inequalities for isolated eigenvalues of $H$ on several complex subplanes.
\end{abstract}
\vspace{3mm}


\section{Introduction}
Let $d\ge 1$ be a dimension of Euclidean space.
We consider the non-self-adjoint Schr\"{o}dinger operator defined as the quasi-maximal accretive operator \cite{Ka} acting in $L^2(\mathbb R^d)$:
\[
H:=H_0+V,\quad H_0:=-\Delta
\]
where the Laplacian $\Delta:=\sum_{j=1}^d\partial^2/\partial x_j^2$ means the distributional derivative and $V$ is the dilation analytic complex potential (see Definition \ref{df:DAC} for detailed definitions).
We define the domain ${\cal D}(H_0)$ of $H_0$ as the second-order Sobolev space $H^2(\mathbb R^d):=W^{2,2}(\mathbb R^d)$.
The $L^2$-inner product and $L^2$-norm are defined by
\[
(u,v):=\int_{\mathbb R^d}u(x)\overline{v(x)}\,{\rm d}x,\quad 
\|u\|_{L^2(\mathbb R^d;\mathbb C)}:=(u,u)^{1/2}
\]
respectively.
Moreover, we consider the one-parameter unitary group $\{U(\theta):L^2(\mathbb R^d)\to L^2(\mathbb R^d);\theta\in \mathbb R\}$ defined by
\[
U(\theta)u(x):=e^{d\theta/2}u(e^{\theta}x)
\]
for $u\in L^2(\mathbb R^d)$.
We put
\begin{align}
H(\theta)&:=U(\theta)HU(\theta)^{-1}
=e^{-2\theta}(H_0+e^{2\theta}V_{\theta}) \label{eq:HtUHU-1}\\
V_{\theta}(x)&:=U(\theta)VU(\theta)^{-1}=V(e^{\theta}x)
\end{align}
and call $H(\theta)$ (resp. $V_{\theta}$) the {\it dilated Hamiltonian} (resp. {\it dilated potential}).
We also call the transform by $U(\theta)$ such as (\ref{eq:HtUHU-1}) the {\it complex dilation}.
We write 
\begin{align}
\label{eq:wtHt}
\widetilde{H}(\theta):=H_0+e^{2\theta}V_{\theta}.
\end{align}
It is of course that $H(0)=\widetilde{H}(0)=H$.
Furthermore, we denote the real (resp. imaginary) part of $z\in \mathbb C$ by ${\rm Re}\,z$ (resp. ${\rm Im}\,z$).

\begin{df}[\cite{So1}]
\label{df:DAC}
$V$ is called the {\rm dilation analytic complex potential} if it satisfies the followings:
Let $d,\gamma \ge 1$.
\begin{enumerate}
\item[i)] $V$ is the multiplication operator with the complex-valued measurable function $\mathbb R^d\ni x\mapsto V(x)\in \mathbb C$ obeying $V\in L^{\gamma+d/2}({\mathbb R}^d;\mathbb C)$.
\item[ii)] $V$ is the $H_0$-compact operator, that is, ${\cal D}(V)\supset {\cal D}(H_0)=H^2({\mathbb R}^d)$ and $V(H_0+1)^{-1}$ is compact in $L^2(\mathbb R^d)$.
\item[iii)] The function $V_{\theta}$ with respect to $\theta\in \mathbb R$ has an analytic continuation 
into the complex strip
\[
\mathscr{S}_{\alpha}:=\{z\in \mathbb C:|{\rm Im}\, z|<\alpha\}
\]
for some $\alpha>0$ as an $L^{\gamma+ d/2}(\mathbb R^d;\mathbb C)$-valued function with respect to $x$. 
\item[iv)] The function $V_\theta(H_0+1)^{-1}$ with respect to $\theta\in {\mathbb R}$ can be extended to $\mathscr{S}_\alpha$ as a ${\bf B}(L^2({\mathbb R}^d))$-valued analytic function, where ${\bf B}(S)$ denotes the set of bounded, everywhere defined operators on the space $S$.
\end{enumerate}
We write the set of dilation analytic complex potentials by ${\bf D}(\mathscr{S}_{\alpha};\mathbb C)$ for convenience.
\end{df}

Since $U(\theta+\phi)$ and $U(\theta)$ are unitarily equivalent for any $\phi\in \mathbb R$, we can suppose that $\theta$ is the pure-imaginary number by setting $\phi=-{\rm Re}\,\theta$.
In other words, $H(\theta)$ does not depend on ${\rm Re}\,\theta$ and $\sigma(H(\theta))$ is only dependent on ${\rm Im}\,\theta$.
$H(\theta)$ is a Kato's type-(A) function (e.g. \cite{CFKS,Ka,RS4}) which is operator-valued and analytic with respect to $\theta\in \mathscr{S}_{\alpha}$.

{\small
\begin{rem}
\begin{itemize}
\item[(1)] The dilation analytic method originally introduced in \cite{AC} and it was defined for real potentials.
We also call the dilation analytic method the {\it complex dilation method} or {\it complex scaling method}.
This method and the now famous results derived by it were organized and customized in e.g. \cite{CFKS,RS4}.
Aguilar and Combes originally proposed dilation analytic potentials so as to give a sufficient condition for the absence of the singularly continuous spectrum of the Schr\"{o}dinger operator (then, remark that the non-negative half line $[0,\infty)$ is the essential spectrum of it).
More to say, the dilation analytic method is a natural factor that we consider and introduce complex potentials.
\item[(2)] $V_{\theta}$ has an analytic extension from $\mathscr{S}_{\alpha}$ to the closure $\overline{\mathscr{S}_{\alpha}}$ of $\mathscr{S}_{\alpha}$ and $H(\theta)$ can be extended from $\mathbb R$ to $\overline{\mathscr{S}_{\alpha}}$ with respect to $\theta$ as a ${\bf B}(L^2({\mathbb R}^d))$-valued analytic function, but we do not need such assumptions in the present paper.
\end{itemize}
\end{rem}
}

\subsection{Complex Lieb--Thirring Type Inequalities}
Throughout the present paper, we write $\sigma(T)$, $\sigma_{{\rm d}}(T)$, $\sigma_{{\rm ess}}(T)$ for the spectrum, discrete spectrum, essential spectrum of the closed operator $T$ respectively.
Also, `isolated eigenvalues' are simply abbreviated as `eigenvalues'.
The algebraic multiplicity $m_{\lambda}(H)$ of $\lambda\in \sigma_{{\rm d}}(H)$ is defined by
\[
m_{\lambda}(H):=\sup_{N\in \mathbb N}\left(\dim \ker(H-\lambda)^N\right).
\]
In estimating the sum of power of eigenvalues hereafter, we count the number of eigenvalues according to their algebraic multiplicities whether potentials are real or complex.

If $V$ decays at infinity, it is well known that $\sigma_{{\rm d}}(H)\subset (-\infty,0)$. 
Then, the Lieb--Thirring inequality for such a real potential $V\in L^{\gamma+d/2}({\mathbb R}^{d};\mathbb R)$ is well known (e.g. \cite{L,LS,LT}) as the estimate on negative eigenvalues:
\begin{align}
\label{eq:rLT}
\sum_{\lambda\in \sigma_{{\rm d}}(H)\subset (-\infty,0)}|\lambda|^{\gamma}\le L_{\gamma,d}\|V_-\|_{L^{\gamma+d/2}({\mathbb R}^{d};\mathbb R)}^{\gamma+d/2},\quad 
V_{\pm}:=\frac{|V|\pm V}{2}
\end{align}
where the dimension $d$ obeys that
\begin{align}
\label{eq:dgamma}
\left\{
\begin{array}{ll}
\displaystyle \gamma \ge 1/2&\ {\rm if}\ d=1,\vspace{1.5mm}\\
\gamma>0&\ {\rm if}\ d=2,\vspace{1mm}\\
\gamma\ge 0&\ {\rm if}\ d\ge 3.
\end{array}
\right.
\end{align}
Then, $L_{\gamma,d}$ is a constant depending on $d,\gamma$ and it is important for the accuracy of the estimate (see e.g. \cite{DLL,LW1,LW2,LL}).
In particular, (\ref{eq:rLT}) is well known as Cwikel--Lieb--Rozenbljum inequalities (e.g. \cite{RS4,Y}) which are estimates on the number of negative eigenvalues of $H$ if $d\ge 3$.
Related to this, Frank, Laptev, Lieb and Seiringer \cite{FLLS} gave some Lieb--Thirring type inequalities for isolated eigenvalues of Schr\"{o}dinger operators with any complex potentials on partial complex planes.
The following inequality (\ref{eq:FLLS}) is particularly the most fundamental result for complex Lieb--Thirring inequalities.

\begin{thm}[\cite{FLLS}]
\label{thm:FLLS}
Let $d,\gamma\ge 1$.
Suppose $V\in L^{\gamma+d/2}({\mathbb R}^{d};\mathbb C)$.
We denote
\begin{align}
\label{eq:Cpmk}
\mathscr{C}_{\pm}(\kappa):=\{z\in \mathbb C:|{\rm Im}\, z|<\pm \kappa {\rm Re}\, z\},
\end{align}
where these sets represent two sets, one for the upper sign and the other for the lower sign. 
Then, for any $\kappa>0$,
\begin{align}
\label{eq:FLLS}
\sum_{\lambda\in \sigma_{{\rm d}}(H)\cap \mathscr{C}_{+}(\kappa)^{{\rm c}}}|\lambda|^{\gamma}\le C_{\gamma,d}\left(1+\frac{2}{\kappa}\right)^{\gamma+d/2}\|V\|_{L^{\gamma+d/2}({\mathbb R}^{d};\mathbb C)}^{\gamma+d/2}
\end{align}
and
\begin{align}
\label{eq:FLLS'}
\sum_{\lambda\in \sigma_{{\rm d}}(H)\cap \mathscr{C}_{-}(\kappa)}|\lambda|^{\gamma}\le \left(1+\kappa\right)L_{\gamma,d}\|({\rm Re}\,V)_-\|_{L^{\gamma+d/2}({\mathbb R}^{d};\mathbb C)}^{\gamma+d/2}.
\end{align}
Here $S^{{\rm c}}$ is the complement set of the set $S$,
\[
C_{\gamma,d}:=2^{1+\gamma/2+d/4}L_{\gamma,d}
\]
and $L_{\gamma,d}$ the constant of real Lieb--Thirring inequalities (\ref{eq:rLT}).
\end{thm}

We can obtain the usual Lieb--Thirring inequality (\ref{eq:rLT}) by letting $\kappa\downarrow 0$ in (\ref{eq:FLLS'}).
In other words, (\ref{eq:FLLS'}) is an inequality which extends (\ref{eq:rLT}).
On the other hand, the Lieb--Thirring inequality for the eigenvalues on the complex left-half plane immediately holds from (\ref{eq:FLLS}) by letting $\kappa\to \infty$.

\begin{cor}[\cite{FLLS}]
\label{cor:FLLS}
Suppose $V\in L^{\gamma+d/2}({\mathbb R}^{d};\mathbb C)$.
For $d,\gamma\ge 1$, one has
\begin{align}
\label{eq:FLLS''}
\sum_{\lambda\in \sigma_{{\rm d}}(H)\cap \{z\in \mathbb C:{\rm Re}\,z\le 0\}}|\lambda|^{\gamma}
\le C_{\gamma,d}\|V\|_{L^{\gamma+d/2}(\mathbb R^d;\mathbb C)}^{\gamma+d/2}.
\end{align}
\end{cor}

{\small
\begin{rem}
\begin{itemize}
\item[(1)] It is now known \cite{B} that complex Lieb--Thirring estimates on all eigenvalues in $\mathbb C\setminus [0,\infty)$ of $H$ with {\it any} complex potential like (\ref{eq:rLT}) cannot hold if $\gamma>d/2$.
\item[(2)] The proofs by \cite{FLLS} of Theorem \ref{thm:FLLS} and Corollary \ref{cor:FLLS} enable us to replace $H=-\Delta+V$ by $H(A):=(-i\nabla+A)^2+V$ with any real vector potential $A$ and complex potential $V$.
(So, we can read Theorem \ref{thm:So} and Theorem \ref{thm:23CLT}-\ref{thm:Reso} described later as results for $H(A)$.)
In addition, their proofs also enable us to replace $|V(x)|$ in (\ref{eq:FLLS}) and (\ref{eq:FLLS''}) by $\{({\rm Re}\,V(x))_-+|{\rm Im}\,V(x)|\}/\sqrt{2}$.
See \cite{FLLS} for details.
\end{itemize}
\end{rem}
}

\cite{So1} shows that, if $V$ is a dilation analytic complex potential, we can obtain the Lieb--Thirring {\bf type} inequality for all eigenvalues (in $\mathbb C\setminus [0,\infty)$) of $H$ as follows. 
On and after, we write $i:=\sqrt{-1}$.

\begin{thm}[\cite{So1}]
\label{thm:So}
Suppose that $V\in {\bf D}(\mathscr{S}_{\alpha};\mathbb C)$ with $\alpha>\pi/4$.
For $d,\gamma\ge 1$, one has
\[
\sum_{\lambda\in \sigma_{{\rm d}}(H)}|\lambda|^{\gamma}
\le C_{\gamma,d}\sum_{\pm}\|V_{\pm i\pi/4}\|_{L^{\gamma+d/2}(\mathbb R^d;\mathbb C)}^{\gamma+d/2}.
\]
More precisely, if we write $\mathbb C_+$ (resp. $\mathbb C_-$) for the upper-half (resp. lower-half) complex plane, we have
\begin{itemize}
\item[1)] the estimate on eigenvalues on $\mathbb C_+$:
\[
\sum_{\lambda\in \sigma_{{\rm d}}(H)\cap(\mathbb C_+\cup (-\infty,0))}|\lambda|^{\gamma}\le C_{\gamma,d}\|V_{i\pi/4}\|_{L^{\gamma+d/2}(\mathbb R^d;\mathbb C)}^{\gamma+d/2},
\]
\item[2)] the estimate on eigenvalues on $\mathbb C_-$:
\[
\sum_{\lambda\in \sigma_{{\rm d}}(H)\cap(\mathbb C_-\cup (-\infty,0))}|\lambda|^{\gamma}\le C_{\gamma,d}\|V_{-i\pi/4}\|_{L^{\gamma+d/2}(\mathbb R^d;\mathbb C)}^{\gamma+d/2}.
\]
\end{itemize}
\end{thm}

The above theorems and corollary indicate that improving $L_{\gamma,d}$ is an important study to increase the accuracy of Lieb--Thirring estimates for complex potentials.

\subsection{Lieb--Thirring Type Estimates on Resonance Eigenvalues}

\begin{df}[e.g. \cite{RS4}]
For $\theta\in \mathscr{S}_{\alpha}$, elements of the complex subset
\[
\sigma_{{\rm res}}(H|\theta):=\sigma_{{\rm d}}(H(\theta))\setminus \sigma_{{\rm d}}(H)
\]
are called {\rm resonance eigenvalues} of $H$ under complex dilation with $\theta\in \mathscr{S}_{\alpha}$.
\end{df}

\begin{rem}
Resonance eigenvalues of $H$ are sometimes defined as isolated and non-real eigenvalues of $H(\theta)$.
We can find that definition in \cite{Sk} for instance.
\end{rem}

One of the reasons to study the eigenvalue estimates for complex potentials is to estimate the resonance eigenvalues or those number.
There is for instance a preceding result on resonance estimates in \cite{BO}, Proposition 6.

\section{Main Results and Proofs}
We will prove Theorem \ref{thm:Reso} which is our main theorem in the same way as the proof of Theorem \ref{thm:So}.
For that reason, we recall some results in \cite{So1}.
Hereafter, $\mathbb C^{+}$ (resp. $\mathbb C^-$) denotes the complex upper-half (resp. lower-half) plane.
Moreover, we always write in what follows $\theta$ (resp. $\phi$) for a complex (resp. real) angle expressed in radians.

\begin{prop}[\cite{So1}]
\label{prop:essSHip}
Suppose $V\in {\bf D}(\mathscr{S}_{\alpha};\mathbb C)$.
Then, 
\begin{align*}
\sigma_{{\rm ess}}(\widetilde{H}(i\phi))&=\sigma_{{\rm ess}}(H_0)=[0,\infty), \\
\sigma_{{\rm ess}}(H(i\phi))&=e^{-2i\phi}[0,\infty)
\end{align*}
for any $i\phi\in \mathscr{S}_{\alpha}$.
\end{prop}

\begin{lem}[\cite{So1}]
\label{lem:Some1}
Suppose that $V\in {\bf D}(\mathscr{S}_{\alpha};\mathbb C)$.
Then, 
\[
\sigma_{{\rm d}}(H)\cap \mathbb C^{\pm}=\sigma_{{\rm d}}(H(i\phi))\cap \mathbb C^{\pm}
\]
for any $i\phi\in \mathscr{S}_{\alpha}\cap \mathbb C^{\pm}$, where the two symbols $\pm$ correspond arbitrarily to each other.
\end{lem}

We write $\{\lambda(i\phi)\}$ for the eigenvalues of $H(i\phi)$.
It is shown (e.g. \cite{Ka,RS4,So1}) that each $\lambda(i\phi)\in \sigma_{{\rm d}}(H(i\phi))$ is given by the branch of one or several analytic functions as Puiseux series. 
Then, they can be written as
\begin{align}
\label{eq:lpwlp}
\widetilde{\lambda}(i\phi)=e^{2i\phi}\lambda(i\phi)
\end{align}
by virtue of (\ref{eq:HtUHU-1}) and (\ref{eq:wtHt}), if $\{\widetilde{\lambda}(i\phi)\}$ are eigenvalues of $\widetilde{H}(i\phi)$.

\begin{lem}[\cite{So1}]
\label{lem:Some2}
Suppose that $V\in {\bf D}(\mathscr{S}_{\alpha};\mathbb C)$.
Then,  
\begin{align}
\label{eq:mlH}
m_{\lambda}(H)=m_{\lambda(i\phi)}(H(i\phi))
=m_{\widetilde{\lambda}(i\phi)}(\widetilde{H}(i\phi))
\end{align}
for any $i\phi \in \mathscr{S}_{\alpha}\cap \mathbb C^{\pm}$.
\end{lem}

{\small
\begin{rem}
\label{rem:CLTDACP}
\begin{itemize}
\item[(1)] It is well known \cite{Ka,RS4} that Proposition \ref{prop:essSHip} and Lemma \ref{lem:Some1} hold for real $V$.
Moreover Lemma \ref{lem:Some2} is the same.
\item[(2)] As we can see from the proofs in \cite{So1}, Lemma \ref{lem:Some1}-\ref{lem:Some2} still hold even if `$\mathbb C^{\pm}$' is replaced with `any subset of $\mathbb C^{\pm}$' in each statement.
However, in order to replace `$\mathbb C^{\pm}$' by `(any subset of) the left-half complex plane', we must keep in mind the range of $\alpha$ (see Theorem \ref{thm:23CLT} and that proof for details).
\end{itemize}
\end{rem}
}

\subsection{New Complex Lieb--Thirring Type Estimates}
Let us give an important theorem.
The following result serves as a lemma to prove Theorem \ref{thm:Reso}.

\begin{thm}
\label{thm:23CLT}
Let $d,\gamma\ge 1$.
Suppose that $V\in {\bf D}(\mathscr{S}_{\alpha};\mathbb C)$ with $\alpha>\frac{\pi}{4}-\frac{1}{2}{\rm Arctan}\,\kappa$ for any $\kappa>0$.
Then, one has
\begin{align}
\label{eq:lamUk}
\sum_{\lambda\in \sigma_{{\rm d}}(H)\cap \mathscr{U}_{\pm \kappa}}|\lambda|^{\gamma}
\le (1+\kappa)L_{\gamma,d}\left\|\left[{\rm Re}\left(e^{\pm i(\frac{\pi}{2}-{\rm Arctan}\,\kappa)}V_{\pm i(\frac{\pi}{4}-\frac{1}{2}{\rm Arctan}\,\kappa)}\right)\right]_-\right\|_{L^{\gamma+d/2}(\mathbb R^d;\mathbb C)}^{\gamma+d/2}
\end{align}
where these represent two inequalities, one for the upper sign and the other for the lower sign, and
\begin{align*}
\mathscr{U}_{\kappa}&:=\{z\in \mathbb C:\pi/2<\arg z<\pi/2+2{\rm Arctan}\,\kappa\}, \\
\mathscr{U}_{-\kappa}&:=\{z\in \mathbb C:-\pi/2-2{\rm Arctan}\,\kappa<\arg z<-\pi/2\}.
\end{align*}
\end{thm}

\begin{proof}
Fix $\kappa>0$ arbitrarily.
We prove only for eigenvalues in $\mathscr{U}_{\kappa}$.
The same can be said for them in $\mathscr{U}_{-\kappa}$.
We write $\lambda$ for an eigenvalue of $H$ and denote the complex left-half plane by $\mathbb C_{<}$.
We can first show that $\lambda(i\phi)=\lambda$ for any $i\phi\in \mathscr{S}_{\alpha}\cap \mathbb C_{<}$ as well as Lemma \ref{lem:Some1}.
We can next show, from (\ref{eq:lpwlp}), that $\widetilde{\lambda}(i\phi)=e^{2i\phi}\lambda$ for any $i\phi\in \mathscr{S}_{\alpha}\cap \mathbb C_{<}$.
We can also see (\ref{eq:mlH}) for any $i\phi\in \mathscr{S}_{\alpha}\cap \mathbb C_{<}$ as well as Lemma \ref{lem:Some2}.
Thus, we should estimate $\{\widetilde{\lambda}(i\phi)\}$ instead of $\{\lambda\}$, because of these facts and (\ref{eq:lpwlp}). 
Let us set $\phi=\frac{\pi}{4}-\frac{1}{2}{\rm Arctan}\,\kappa$. 
It follows, from the above, that
\[
e^{i(2\phi)}\left(\sigma_{{\rm d}}(H)\cap \mathscr{U}_{\kappa}\right)
=\sigma_{{\rm d}}(\widetilde{H}(i\phi))\cap \mathscr{C}_{-}(\kappa)
\]
by recalling (\ref{eq:Cpmk}) for $\mathscr{C}_{-}(\kappa)$.
So, we have
\begin{align*}
\sum_{\lambda\in \sigma_{{\rm d}}(H)\cap \mathscr{U}_{\kappa}}|\lambda|^{\gamma}
&=\sum_{\lambda\in \sigma_{{\rm d}}(H)\cap \mathscr{U}_{\kappa}}|e^{2i\phi}\lambda|^{\gamma} \\
&=\sum_{\widetilde{\lambda}(i\phi)\in \sigma_{{\rm d}}(\widetilde{H}(i\phi))\cap \mathscr{C}_{-}(\kappa)}|\widetilde{\lambda}(i\phi)|^{\gamma} \\
&\le (1+\kappa)L_{\gamma,d}\|[{\rm Re}(e^{i(2\phi)}V_{i\phi})]_{-}\|_{L^{\gamma+d/2}(\mathbb R^d;\mathbb C)}^{\gamma+d/2}.
\end{align*}
Hence, this completes the proof.
\end{proof}

We write $\mathbb C_{{\rm II}}$ (resp. $\mathbb C_{{\rm III}}$) for the second (resp. third) quadrant of $\mathbb C$.
Because of Theorem \ref{thm:23CLT}, we can easily know Lieb--Thirring type inequalities for eigenvalues on $\mathbb C_{{\rm II}}$ or $\mathbb C_{{\rm III}}$ as follows.

\begin{cor}
\label{cor:LTt23}
Let $d,\gamma\ge 1$.
Suppose that $V\in {\bf D}(\mathscr{S}_{\alpha};\mathbb C)$ with $\alpha>\pi/8$.
Then, 
\begin{itemize}
\item[1)] Eigenvalue estimate on $\mathbb C_{{\rm II}}$:
\[
\sum_{\lambda\in \sigma_{{\rm d}}(H)\cap \mathbb C_{{\rm II}}}|\lambda|^{\gamma}
\le 2L_{\gamma,d}\left\|\left[{\rm Re}(e^{i\pi/4}V_{i\pi/8})\right]_-\right\|_{L^{\gamma+d/2}(\mathbb R^d;\mathbb C)}^{\gamma+d/2},
\]
\item[2)] Eigenvalue estimate on $\mathbb C_{{\rm III}}$:
\[
\sum_{\lambda\in \sigma_{{\rm d}}(H)\cap \mathbb C_{{\rm III}}}|\lambda|^{\gamma}
\le 2L_{\gamma,d}\left\|\left[{\rm Re}(e^{-i\pi/4}V_{-i\pi/8})\right]_-\right\|_{L^{\gamma+d/2}(\mathbb R^d;\mathbb C)}^{\gamma+d/2}.
\]
\end{itemize}
\end{cor}

\begin{proof}
It is obvious from (\ref{eq:lamUk}), since we have $\kappa=1$ by setting ${\rm Arctan}\,\kappa=\pi/4$.
\end{proof}

\subsection{Estimates on Complex Resonance Eigenvalues for Complex Potentials}
We now would like to estimate the complex eigenvalues which appear newly by complex dilation.
We focus on eigenvalues of $H(i\phi)$ appear in open complex sector $\{z\in \mathbb C:-2\phi<\arg z<0\}$.
For convenience, let us call them {\it complex resonance eigenvalues of ${\rm H}$} hereinafter.
The following result is our main theorem.
The idea of that proof is the way which can be called `double complex dilation.'
We denote
\begin{align*}
V_{\theta_1,\ldots,\theta_n}(x)&:=[U(\theta_n)\cdots U(\theta_2)U(\theta_1)VU(\theta_1)^{-1}U(\theta_2)^{-1}\cdots U(\theta_n)^{-1}](x) \\
&\,=V_{\theta_1+\cdots+\theta_n}(x), \\
H(\theta_1,\ldots,\theta_n)&:=U(\theta_n)\cdots U(\theta_2)U(\theta_1)HU(\theta_1)^{-1}U(\theta_2)^{-1}\cdots U(\theta_n)^{-1} \\
&\,=H(\theta_1+\cdots+\theta_n)
\end{align*}
for any $n\in \mathbb N$.
The same applies to $\widetilde{H}(\theta_1,\ldots,\theta_n)$.

\begin{thm}
\label{thm:Reso}
Let $d,\gamma\ge 1$.
Suppose that $V\in {\bf D}(\mathscr{S}_{\alpha};\mathbb C)$ with $\alpha>|\frac{3}{2}\phi-\frac{\pi}{2}|$.
Then, complex resonance eigenvalues of $H$ are estimated as
\begin{align}
\label{eq:Resonance}
\sum_{\mu\in \sigma_{{\rm res}}(H|i\phi)\setminus [0,\infty)}|\mu|^\gamma
\le (1+\tan\phi)L_{\gamma,d}\left\|\left[{\rm Re}\left(e^{i(\phi-\frac{\pi}{2})}V_{i(\frac{3}{2}\phi-\frac{\pi}{2})}\right)\right]_-\right\|_{L^{\gamma+d/2}(\mathbb R^d;\mathbb C)}^{\gamma+d/2}
\end{align}
for any $i\phi\in \mathscr{S}_{\alpha}$.
\end{thm}

\begin{proof}
This proof is similar to the proofs of Theorem \ref{thm:So} and Theorem \ref{thm:23CLT}.
The key to proof is to apply Theorem \ref{thm:23CLT} as $\kappa=\tan \phi$.
Then, Lemma \ref{lem:Some1}-\ref{lem:Some2} and (\ref{eq:lpwlp}) imply that
\begin{align*}
e^{-i\pi/2}\sigma_{{\rm res}}(H|i\phi)&=e^{-i\pi/2}\Bigl[\sigma_{{\rm d}}(H(i\phi))\cap \{z\in \mathbb C:-2\phi<\arg z<0\}\Bigr] \\
&=\sigma_{{\rm d}}\bigl(\widetilde{H}(i\phi,-i\pi/4)\bigr)\cap \mathscr{U}_{-\tan \phi} \\
&=\sigma_{{\rm d}}\bigl(\widetilde{H}\bigl(i(\phi-\tfrac{\pi}{4})\bigr)\bigr)\cap \mathscr{U}_{-\tan \phi}.
\end{align*}
Hence, it follows that
\begin{align*}
\sum_{\mu\in \sigma_{{\rm res}}(H|i\phi)}|\mu|^\gamma&\le \sum_{\widetilde{\mu}(i(\phi-\tfrac{\pi}{4}))\in \sigma_{{\rm d}}(\widetilde{H}(i(\phi-\tfrac{\pi}{4})))\cap \mathscr{U}_{-\tan \phi}}|\widetilde{\mu}(i(\phi-\tfrac{\pi}{4}))|^{\gamma} \\
&\le (1+\tan \phi)L_{\gamma,d}\left\|\left[{\rm Re}\left(e^{-i(\frac{\pi}{2}-\phi)}V_{i(\phi-\tfrac{\pi}{4}),-i(\tfrac{\pi}{4}-\tfrac{\phi}{2})}\right)\right]_-\right\|_{L^{\gamma+d/2}(\mathbb R^d;\mathbb C)}^{\gamma+d/2} \\
&=(1+\tan\phi)L_{\gamma,d}\left\|\left[{\rm Re}\left(e^{i(\phi-\frac{\pi}{2})}V_{i(\frac{3}{2}\phi-\frac{\pi}{2})}\right)\right]_-\right\|_{L^{\gamma+d/2}(\mathbb R^d;\mathbb C)}^{\gamma+d/2}
\end{align*}
by applying (\ref{eq:lamUk}) for eigenvalues in $\mathscr{U}_{-\kappa}$.
\end{proof}

{\small
\begin{rem}
\label{rem:emb}
We write $\sigma_{{\rm p}}(T)$ for the point spectrum of the closed operator $T$.
If $V$ is a dilation analytic real potential, the spectral decomposition theorem implies that 
\begin{align}
\label{eq:pp0id0i}
\sigma_{{\rm p}}(H)\cap (0,\infty)=\sigma_{{\rm d}}(H(i\phi))\cap (0,\infty)
\end{align}
for $\phi\in (0,\min\{\alpha,\pi/2\})$ (e.g. \cite{RS4}).
In this sense, embedded eigenvalues (in the essential or absolutely continuous spectrum $[0,\infty)$) of $H$ are invariant under complex dilation.
(In the case of dilation analytic complex potentials, we cannot however use the spectral decomposition theorem and we have no idea if the same is true.)
Thus, all eigenvalues which appear newly by complex dilation belong to $\{z\in \mathbb C:-2\phi<\arg z<0\}$ if embedded eigenvalues of $H$ exist.
Moreover, then, Lemma \ref{lem:Some2} also holds for embedded eigenvalues and the proof is similar.
\end{rem}
}

We derived Corollary \ref{cor:LTt23} by complex dilation, but we can produce the following results by double complex dilation and Corollary \ref{cor:LTt23}. 
Here, $\mathbb C_{{\rm I}}$ (resp. $\mathbb C_{{\rm IV}}$) denotes the first (resp. fourth) quadrant of $\mathbb C$.

\begin{prop}
\label{prop:CIIV}
Let $d,\gamma\ge 1$.
Suppose that $V\in {\bf D}(\mathscr{S}_{\alpha};\mathbb C)$ with $\alpha> 3\pi/8$.
Then, 
\begin{itemize}
\item[1)] Eigenvalue estimate on $\mathbb C_{{\rm I}}$:
\[
\sum_{\lambda\in \sigma_{{\rm d}}(H)\cap \mathbb C_{{\rm I}}}|\lambda|^{\gamma}
\le 2L_{\gamma,d}\left\|\left[{\rm Re}\left(e^{3\pi i/4}V_{3\pi i/8}\right)\right]_-\right\|_{L^{\gamma+d/2}(\mathbb R^d;\mathbb C)}^{\gamma+d/2},
\]
\item[2)] Eigenvalue estimate on $\mathbb C_{{\rm IV}}$:
\[
\sum_{\lambda\in \sigma_{{\rm d}}(H)\cap \mathbb C_{{\rm IV}}}|\lambda|^{\gamma}
\le 2L_{\gamma,d}\left\|\left[{\rm Re}\left(e^{-3\pi i/4}V_{-3\pi i/8}\right)\right]_-\right\|_{L^{\gamma+d/2}(\mathbb R^d;\mathbb C)}^{\gamma+d/2}.
\]
\end{itemize}
\end{prop}

\begin{proof}
We should apply Corollary \ref{cor:LTt23} to $\widetilde{\lambda}(i\pi/4)=i\lambda\in \mathbb C_{{\rm II}}$ if $\lambda\in \mathbb C_{{\rm I}}$.
In fact, we have 1) as follows:
\begin{align*}
\sum_{\lambda\in \sigma_{{\rm d}}(H)\cap \mathbb C_{{\rm I}}}|\lambda|^{\gamma}
&\le 2L_{\gamma,d}\left\|\left[{\rm Re}\left(e^{i\pi/2}e^{i\pi/4}V_{i\pi/8,i\pi/4}\right)\right]_-\right\|_{L^{\gamma+d/2}(\mathbb R^d;\mathbb C)}^{\gamma+d/2} \\
&=2L_{\gamma,d}\left\|\left[{\rm Re}\left(e^{3\pi i/4}V_{3\pi i/8}\right)\right]_-\right\|_{L^{\gamma+d/2}(\mathbb R^d;\mathbb C)}^{\gamma+d/2}.
\end{align*}
2) can be shown in the same way.
\end{proof}

We write $\mathbb C_{>}$ for the right-half complex plane.
Proposition \ref{prop:CIIV} immediately derives the following estimate.

\begin{cor}
Let $d,\gamma\ge 1$.
Suppose that $V\in {\bf D}(\mathscr{S}_{\alpha};\mathbb C)$ with $\alpha>3\pi/8$.
Then, the eigenvalues of $H$ on $\mathbb C_{>}\setminus [0,\infty)$ are estimated as follows:
\begin{align}
\label{eq:C>esti}
\sum_{\lambda\in \sigma_{{\rm d}}(H)\cap \mathbb C_{>}}|\lambda|^{\gamma}
\le 2L_{\gamma,d}\sum_{\pm}\left\|\left[{\rm Re}\left(e^{\pm 3\pi i/4}V_{\pm 3\pi i/8}\right)\right]_-\right\|_{L^{\gamma+d/2}(\mathbb R^d;\mathbb C)}^{\gamma+d/2}.
\end{align}
\end{cor}

Thus, we can obtain an estimate on all eigenvalues different form Theorem \ref{thm:So} as follows.

\begin{cor}[cf. \cite{So1}]
Let $d,\gamma\ge 1$.
Suppose that $V\in {\bf D}(\mathscr{S}_{\alpha};\mathbb C)$ with $\alpha>3\pi/8$.
Then, one has
\begin{align}
\label{eq:allA}
\sum_{\lambda\in \sigma_{{\rm d}}(H)}|\lambda|^{\gamma}\le C_{\gamma,d}\|V\|_{L^{\gamma+d/2}(\mathbb R^d;\mathbb C)}^{\gamma+d/2}+2L_{\gamma,d}\sum_{\pm}\left\|\left[{\rm Re}\left(e^{\pm 3\pi i/4}V_{\pm 3\pi i/8}\right)\right]_-\right\|_{L^{\gamma+d/2}(\mathbb R^d;\mathbb C)}^{\gamma+d/2}.
\end{align}
\end{cor}

\begin{proof}
The desired estimate follows by combining (\ref{eq:FLLS''}) and (\ref{eq:C>esti}).
\end{proof}

\subsection{Estimates on Embedded Eigenvalues for Real Potentials}
We would like to estimate embedded eigenvalues of real Schr\"{o}dinger operators as Lieb--Thirring type.
We assume in this subsection that $H=H_0+V$ with real potentials $V$ have embedded eigenvalues $\{\lambda_{{\rm e}}\}$ in $\sigma_{{\rm ac}}(H)=[0,\infty)$.
Moreover, we write in this subsection $H_{{\rm c}}$ (resp. $H_{{\rm r}}$) for the Schr\"{o}dinger operator with the complex (resp. real) potential to prevent cunfusion.

We first derive the following estimate on isolated eigenvalues on the upper imaginary axis $i\mathbb R_{+}:=\{z\in \mathbb C:{\rm Re}\,z=0,\ {\rm Im}\,z>0\}$ of complex Schr\"{o}dinger operators.

\begin{thm}
Let $d,\gamma\ge 1$.
Suppose that $V\in {\bf D}(\mathscr{S}_{\alpha};\mathbb C)$ with $\alpha>\pi/4$.
Then, the eigenvalues of $H_{{\rm c}}$ on $i\mathbb R_{+}$ are estimated as follows:
\begin{align}
\label{eq:CLTiR+}
\sum_{\lambda\in \sigma_{{\rm d}}(H_{{\rm c}})\cap i\mathbb R_{+}}|\lambda|^{\gamma}
\le L_{\gamma,d}\|({\rm Im}\,V_{i\pi /4})_+\|_{L^{\gamma+d/2}(\mathbb R^d;\mathbb C)}^{\gamma+d/2}.
\end{align}
\end{thm}

\begin{proof}
We take eigenvalues $\{\lambda\}$ of $H_{{\rm c}}$ on $i\mathbb R_{+}$.
We obtain $\{\lambda\}=\{\lambda(i\phi)\}$ by virtue of Lemma \ref{lem:Some1}, and $\{\widetilde{\lambda}(i\phi)\}=\{e^{i(2\phi)}\lambda(i\phi)\}$ by virtue of (\ref{eq:lpwlp}).
So, setting $\phi=\pi/4$ implies that $\{\widetilde{\lambda}(i\pi/4)\}=\{e^{i\pi/2}\lambda(i\pi/4)\}=\{i\lambda\}$ and all of $i\lambda$ lie on $\mathbb R_{-}$.
Also, these $i\lambda$ correspond to negative real eigenvalues under the potential $iV_{i\pi/4}$.
Thus, the standard Lieb--Thirring inequality implies that
\begin{align*}
\sum_{\lambda\in \sigma_{{\rm d}}(H_{{\rm c}})\cap i\mathbb R_{+}}|\lambda|^{\gamma}
&=\sum_{\widetilde{\lambda}(i\pi/4)\in \sigma_{{\rm d}}(H_{{\rm r}})}|\widetilde{\lambda}(i\pi/4)|^{\gamma} \\
&\le L_{\gamma,d}\|{\rm Re}(iV_{i\pi/4})_-\|_{L^{\gamma+d/2}({\mathbb R}^{d};\mathbb R)}^{\gamma+d/2} \\
&=L_{\gamma,d}\|({\rm Im}\,V_{i\pi /4})_+\|_{L^{\gamma+d/2}(\mathbb R^d;\mathbb C)}^{\gamma+d/2}.
\end{align*}
Here the above estimate is discussed with the property of multiplicities of eigenvalues: Lemma \ref{lem:Some2}.
This completes the proof.
\end{proof}

We next prove the following estimate by double complex dilation in the same way as Theorem \ref{thm:Reso}.

\begin{thm}
Let $d,\gamma\ge 1$.
Suppose that $V\in {\bf D}(\mathscr{S}_{\alpha};\mathbb R)$ with $\alpha>\pi/2$.
Then, the embedded eigenvalues of $H_{{\rm r}}$ in $[0,\infty)$ are estimated as follows:
\begin{align}
\sum_{\lambda_{{\rm e}}\in \sigma_{{\rm pp}}(H_{{\rm r}})\cap [0,\infty)}\lambda_{{\rm e}}^{\gamma}
\le L_{\gamma,d}\|({\rm Re}\,V_{i\pi /2})_{+}\|_{L^{\gamma+d/2}(\mathbb R^d;\mathbb R)}^{\gamma+d/2}.
\end{align}
\end{thm}

\begin{proof}
It is sufficient to consider positive embedded eigenvalues $\{\lambda_{{\rm e}}\}$ of $H_{{\rm r}}$, because $\sum_{\lambda_{{\rm e}}\in [0,\infty)}|\lambda_{{\rm e}}|^{\gamma}=\sum_{\lambda_{{\rm e}}\in (0,\infty)}|\lambda_{{\rm e}}|^{\gamma}$.
We obtain $\{\lambda_{{\rm e}}\}=\{\lambda(i\phi)\}$ by virtue of (\ref{eq:pp0id0i}), and $\{\widetilde{\lambda}(i\phi)\}=\{e^{i(2\phi)}\lambda(i\phi)\}$ by virtue of (\ref{eq:lpwlp}).
So, setting $\phi=\pi/4$ implies that $\{\widetilde{\lambda}(i\pi/4)\}=\{i\lambda_{{\rm e}}\}$ and all of $i\lambda_{{\rm e}}$ lie on $i\mathbb R_{+}$.
Also, these $\lambda_{{\rm e}}$ correspond to purely imaginary eigenvalues under the potential $iV_{i\pi/4}$.
Thus, (\ref{eq:CLTiR+}) implies that
\begin{align*}
\sum_{\lambda\in \sigma_{{\rm pp}}(H_{{\rm r}})\cap [0,\infty)}\lambda_{{\rm e}}^{\gamma}
&=\sum_{\widetilde{\lambda}(i\pi/4)\in \sigma_{{\rm d}}(H_{{\rm r}})}|\widetilde{\lambda}(i\pi/4)|^{\gamma} \\
&\le L_{\gamma,d}\|{\rm Im}(iV_{i\pi/4,i\pi/4})_+\|_{L^{\gamma+d/2}({\mathbb R}^{d};\mathbb R)}^{\gamma+d/2} \\
&=L_{\gamma,d}\|({\rm Re}\,V_{i\pi/2})_+\|_{L^{\gamma+d/2}(\mathbb R^d;\mathbb C)}^{\gamma+d/2}.
\end{align*}
Here the above estimate is discussed with the property of multiplicities of (positive) embedded eigenvalues: Lemma \ref{lem:Some2} and Remark \ref{rem:emb}.
Hence we have gained the proof.
\end{proof}

\section{Appendix}
We are interested in how the complex dilation affects the accuracy of eigenvalue estimates.
In this appendix, we investigate the $L^{p}$-norms of dilated potentials via examples.

\subsection{$\|V\|_{L^{\gamma+d/2}(\mathbb R^d;\mathbb C)}$ v.s. $\|V_{i\phi}\|_{L^{\gamma+d/2}(\mathbb R^d;\mathbb C)}$}
We first argue the comparison of values of $\|V\|_{L^2(\mathbb R^d;\mathbb C)}$ and $\|V_{i\phi}\|_{L^2(\mathbb R^d;\mathbb C)}$.
Recall that the real potential $V$ which belongs to $L^p(\mathbb R^d)$ for $p>\max\{2-\varepsilon,d/2\}$ with any $\varepsilon>0$ is $H_0$-compact.
That is, we only need to show that $V\in L^p(\mathbb R^d)$ if $p\ge 2$ and $p>d/2$
in order to verify that $V$ is $H_0$-compact, as is well known.

We feel that complex dilation may increase the norm of the potential in general. 
(One of such examples can be actually seen in \cite{So1}. 
See also Proposition \ref{prop:merV} that will be mentioned later.)
However, the following example gives us that our feeling is {\bf not} always true.

\begin{prop}
Let $d\ge 1$ and $\gamma\ge \max\{2-d/2,1\}$.
Suppose that the potential $V$ is defined as a multiplication operator with a Gauss-type function
\begin{align}
\label{eq:Gausstf}
V(x)=e^{-cx^2},\quad c\in \{z\in \mathbb C:{\rm Re}\,z>0\}
\end{align}
 on $\mathbb R^d$.
Then, $V\in {\bf D}(\mathscr{S}_{\alpha};\mathbb C)$ for any $i\phi\in \mathscr{S}_{\alpha}$ obeying
\begin{align}
\label{eq:cphi}
({\rm Re}\,c)\cos 2\phi>({\rm Im}\, c)\sin 2\phi,
\end{align}
and the followings hold:
\begin{itemize}
\item[1)] If ${\rm Re}\,c\ge ({\rm Re}\,c)\cos 2\phi-({\rm Im}\, c)\sin 2\phi$, then one has
\[
\|V\|_{L^{\gamma+1/2}(\mathbb R^d;\mathbb C)}\le \|V_{i\phi}\|_{L^{\gamma+1/2}(\mathbb R^d;\mathbb C)}.
\]
\item[2)] If ${\rm Re}\,c\le ({\rm Re}\,c)\cos 2\phi-({\rm Im}\, c)\sin 2\phi$, then one has
\[
\|V_{i\phi}\|_{L^{\gamma+1/2}(\mathbb R^d;\mathbb C)}\le \|V\|_{L^{\gamma+1/2}(\mathbb R^d;\mathbb C)}.
\]
\end{itemize}
\end{prop}

\begin{proof}
It is not difficult to see that $V$ is dilation analytic on $\mathscr{S}_{\alpha}$ for all $\gamma\ge \max\{2-d/2,1\}$ with any $d\ge 1$.
It is however sufficient to prove this proposition for $d=1$ by virtue of the exponential law.
We assume (\ref{eq:cphi}). 
Then, $V_{i\phi}\in L^{\gamma+1/2}(\mathbb R;\mathbb C)$ and we have
\begin{align*}
\|V\|_{L^{\gamma+1/2}(\mathbb R;\mathbb C)}^{\gamma+1/2}&=\int_{-\infty}^{\infty}|e^{-cx^2}|^{\gamma+1/2}\,{\rm d}x=\left(\frac{\pi}{({\rm Re}\,c)(\gamma+1/2)}\right)^{1/2}, 
\end{align*}
\begin{align}
\label{eq:Vipnorm}
\begin{aligned}
\|V_{i\phi}\|_{L^{\gamma+1/2}(\mathbb R;\mathbb C)}^{\gamma+1/2}&=\int_{-\infty}^{\infty}|e^{-c(e^{i\phi}x)^2}|^{\gamma+1/2}\,{\rm d}x \\
&=\left(\frac{\pi}{\{({\rm Re}\,c)\cos 2\phi-({\rm Im}\, c)\sin 2\phi\}(\gamma+1/2)}\right)^{1/2}.
\end{aligned}
\end{align}
Hence, the proof of this theorem completes.
\end{proof}

\subsection{On Monotonicity of $\|V_{i\phi}\|_{L^{\gamma+d/2}(\mathbb R^d;\mathbb C)}$}
We finally investigate whether $\|V_{i\phi}\|_{L^{\gamma+d/2}(\mathbb R^d;\mathbb C)}$ is monotonic with respect to the dilation angle $\phi$. 
We feel that the more complex dilation we give, the bigger the values of norms of dilation analytic potentials may be.
In fact, we can see an example that affirms our feeling as follows.

\begin{prop}[cf. \cite{So1}]
\label{prop:merV}
Let $d=1$ and $\gamma\ge 3/2$.
We define the potential $V$ as a multiplication operator by 
\[
V(x)=\frac{c}{(1+x^2)^{s}},\quad s>\frac{1}{2\gamma+1},\quad c\in \mathbb C.
\]
Then, $V\in {\bf D}(\mathscr{S}_{\alpha};\mathbb C)$ and $\{\|V_{i\phi}\|_{L^{\gamma+1/2}(\mathbb R;\mathbb C)}\}_{\phi\in [0,\pi/2)}$, $i\phi\in \mathscr{S}_{\alpha}$, is always monotone increasing.
\end{prop}

\begin{proof}
It is easy to see that $V\in L^{\gamma+1/2}(\mathbb R;\mathbb C)$ and $V$ is $H_0$-compact, if $s>1/(2\gamma+1)$ and $\gamma\ge 3/2$.
Since
\begin{align}
\label{eq:Cs4s2g1}
|V_{i\phi}(x)|^{\gamma+1/2}
=\frac{|c|^{\gamma+1/2}}{(x^4+2(\cos 2\phi)x^2+1)^{s(2\gamma+1)}}
\le \frac{C_{\gamma}}{x^{4s(2\gamma+1)}}
\end{align}
for a suitable constant $C_{\gamma}>0$ depending on $\gamma$, we also have $V_{i\phi}\in L^{\gamma+1/2}(\mathbb R;\mathbb C)$ because of $s>1/(2\gamma+1)$.
It is not difficult to see that $V$ is dilation analytic from the above.
We now consider the function $F(\phi):=2x^2\cos 2\phi+(x^4+1)$ with respect to $\phi$ by fixing $x\in \mathbb R$.
Since $F$ is monotone decreasing on $[0,\pi/2)$, the proof of this proposition completes from (\ref{eq:Cs4s2g1}).
\end{proof}

As the above, we feel that the norms of dilated potentials may have monotonically increasing properties.
We can however see that our feeling is {\bf not} always true as follows.

\begin{prop}
Let $d=1$ and $\gamma\ge 3/2$. 
For $V\in {\bf D}(\mathscr{S}_{\alpha};\mathbb C)$ defined as a multiplication operator with a Gauss-type function (\ref{eq:Gausstf}) for any $i\phi\in \mathscr{S}_{\alpha}$ obeying (\ref{eq:cphi}), the followings hold:
\begin{itemize}
\item[1)] If ${\rm Im}\, c>0$, then $\{\|V_{i\phi}\|_{L^{\gamma+1/2}(\mathbb R;\mathbb C)}\}_{\phi}$ is monotone increasing.
\item[2)] If ${\rm Im}\, c<0$ and $\phi \in [0,p)$ (resp. $[p,\pi/2)$), then $\{\|V_{i\phi}\|_{L^{\gamma+1/2}(\mathbb R;\mathbb C)}\}_{\phi}$ is monotone decreasing (resp. monotone increasing).
Here
\begin{align}
\label{eq:cp}
p:=\frac{1}{2}{\rm Arctan}\left(-\frac{{\rm Im}\,c}{{\rm Re}\,c}\right).
\end{align}
\end{itemize}
\end{prop}

\begin{proof}
We consider the function $F(\phi):=({\rm Re}\,c)\cos 2\phi-({\rm Im}\, c)\sin 2\phi$ with respect to $\phi\in [0,\pi/2)$.
Remark ${\rm Re}\,c>0$.
Since we have $F'(\phi)=-2\{({\rm Re}\,c)\sin 2\phi+({\rm Im}\,c)\cos 2\phi\}$, we obtain the critical point $p$ defined by (\ref{eq:cp}) by solving $F'(\phi)=0$.

1) We assume ${\rm Im}\,c>0$. 
Then, $p<0$, 
\begin{align}
F(0)={\rm Re}\,c>0\quad {\rm and} \label{eq:F0}\\ 
\lim_{\phi\uparrow \pi/2}F(\phi)=-a<0. \label{eq:Fphi}
\end{align}
Thus, $F$ is monotone decreasing on $[0,\pi/2)$.
Hence, (\ref{eq:Vipnorm}) implies that $\{\|V_{i\phi}\|_{L^{\gamma+1/2}(\mathbb R;\mathbb C)}\}_{\phi}$ is monotone increasing on $[0,\pi/2)$.

2) We assume ${\rm Im}\,c<0$. 
Then, $p>0$, (\ref{eq:F0}), (\ref{eq:Fphi}) and
\[
F(p)=({\rm Re}\,c)\cos\left({\rm Arctan}\left(-\frac{{\rm Im}\,c}{{\rm Re}\,c}\right)\right)-({\rm Im}\,c)\sin\left({\rm Arctan}\left(-\frac{{\rm Im}\,c}{{\rm Re}\,c}\right)\right)>0
\]
because of ${\rm Re}\,c>0$.
Thus, $F$ is monotone decreasing on $[0,p)$ and is monotone increasing on $[p,\pi/2)$.
Hence, (\ref{eq:Vipnorm}) implies that $\{\|V_{i\phi}\|_{L^{\gamma+1/2}(\mathbb R;\mathbb C)}\}_{\phi}$ is monotone increasing on $[0,p)$ and is monotone decreasing on $[p,\pi/2)$.
\end{proof}
\vspace{3mm}

{\small
\begin{center}
{\sc Acknowledgement}
\end{center}
The author would like to thank referees for giving him suitable and valuable advice.
He also appreciates the researchers who listened to and commented my presentation of the present paper at various conferences.
}

{\small 

\vspace{4mm}

\noindent
{\sc Norihiro Someyama}
\vspace{2mm}

Shin-yo-ji Buddhist Temple, 5-44-4 Minamisenju, Arakawa-ku, Tokyo 116-0003 Japan

E-mail: {\tt philomatics@outlook.jp}

ORCID iD: https://orcid.org/0000-0001-7579-5352
\\

He received a M.Sc. degree from Gakushuin University in 2014 and completed the Ph.D program without a Ph.D. degree the same university in 2017.
He is a head priest of Shin-yo-ji Buddhist Temple in Japan.
His research interests are the spectral theory of Schr\"{o}dinger operators and the theory of Schr\"{o}dinger equations on the fuzzy spacetime.
He received the Member Encouragement Award of Biomedical Fuzzy System Association for his lecture entitled `{\it Characteristic Analysis of Fuzzy Graph and its Application IV}' in November 2018 and the Excellent Presentation Award of National Congress of Theoretical and Applied Mechanics / JSCE Applied Mechanics Symposium for his lecture entitled `{\it Number of Eigenvalues of Non-self-adjoint Schr\"{o}dinger Operators with Dilation Analytic Complex Potentials}' in October 2019.
}
\end{document}